\documentclass[12pt]{amsart}
\usepackage{amssymb}
\usepackage{amscd}
\usepackage{epsfig}
\usepackage{xy}
\usepackage{amsfonts}
\usepackage{amsmath}
\usepackage{amsthm}
\usepackage{times}
\theoremstyle{plain}
\newtheorem{prop}{Proposition}[section]
\newtheorem{thm}[prop]{Theorem}
\newtheorem{cor}[prop]{Corollary}
\newtheorem{lem}[prop]{Lemma}

\theoremstyle{definition}
\newtheorem{defn}[prop]{Definition}
\newtheorem{rem}[prop]{Remark}
\newtheorem{example}[prop]{Example}

\numberwithin{equation}{section}
\newcommand{\F}{\mathbb{F}}
\newcommand{\Z}{\mathbb{Z}}

\begin{document}

\title{Bruhat-Chevalley order on the rook monoid}

\author{Mahir Bilen Can,\\ Lex E. Renner}
\address{University of Western Ontario, Canada}
\email{mcan@uwo.ca, lex@uwo.ca}

\date{}

\begin{abstract}
The rook monoid $R_n$ is the finite monoid whose elements are the $0-1$ matrices with at most one nonzero entry in each row and column. The group of invertible elements of $R_n$ is isomorphic to the symmetric group $S_n$. The natural extension to $R_n$ of the Bruhat-Chevalley ordering on the symmetric group is defined in \cite{Renner86}. In this paper, we find an efficient, combinatorial description of the Bruhat-Chevalley ordering on $R_n$. We also give a useful, combinatorial formula for the length function on $R_n$.
\end{abstract}

\maketitle

\section{\textbf{Introduction}}

Let $GL_n$ be the general linear group over an algebraically closed field $\F$. There is a much-studied decomposition of $GL_n$ into double cosets of the Borel subgroup $B\subset GL_n$ of invertible upper triangular matrices
\begin{equation}\label{E:GBdecomposition}
GL_n = \bigcup_{w\in S_n} B w B,
\end{equation}
where the union is indexed by the symmetric group $S_n$. Elements of $S_n$ are identified with
$0-1$ matrices with exactly one nonzero entry in each row and column.

The decomposition  in (\ref{E:GBdecomposition})  is often refered to as the Bruhat decomposition and it holds, more generally, for reductive groups and reductive monoids (see \cite{PennellPutchaRenner,Renner86}).  In the case of the monoid $M_n$ of $n\times n$ matrices, the Bruhat decomposition is given by
\begin{equation}
M_n = \bigcup_{\sigma \in R_n} B \sigma B,
\end{equation}
where the union is indexed by the rook monoid $R_n$. The elements of $R_n$ are identified with $0-1$ matrices which have at most one nonzero entry in each row and column.

The Bruhat-Chevalley order on $S_n$ is defined in terms of the inclusion relationships between double cosets in (\ref{E:GBdecomposition}). Namely,
if $v, w \in S_n$, then
\begin{equation}\label{E:BCgroup}
v \leq w\ \iff\ Bv B \subseteq \overline{Bw B},
\end{equation}
where the overline stands for the Zariski closure in $GL_n$.

There is a natural extension of this partial order on the rook monoid $R_n$ (see \cite{PennellPutchaRenner, Renner86} for more details).
\begin{equation}\label{E:BCmonoid}
\sigma \leq \tau\ \iff\ B\sigma B \subseteq \overline{B \tau B},
\end{equation}
for $\sigma,\tau\in R_n$.

In \cite{Putcha01}, Putcha describes the partial ordering
(\ref{E:BCmonoid}) for the constant-rank subsets of the rook monoid in terms of the Bruhat order
of related symmetric groups (he describes this partial order, much more generally, for any $J$-class of a Renner monoid).

In \cite{MS05}, using a partial ordering exactly like (\ref{E:BCmonoid}), Miller and Sturmfels study the poset of Zariski closures of $B\times B_+$-orbits on the space of the $k\times l$ matrices. Here $B$ denotes the group of the invertible upper triangular $k\times k$ matrices, and $B_+$ denotes the group of invertible lower triangular $l\times l$ matrices. These $B\times B_+$-orbits are indexed by the
$0-1$, $k\times l$ matrices with at most one nonzero entry in each row and column.

For computational purposes, one would like to have an efficient, combinatorial characterization of the Bruhat-Chevalley ordering on $R_n$. This characterization, in the case of the symmetric group, had been explained to us by V. Deodhar.

\subsubsection{Deodhar's characterization}

For an integer valued vector $a=(a_1,...,a_n)\in \Z^n$, let $\widetilde{a} = (a_{\alpha_1},....,a_{\alpha_n})$  be the rearrangement of the entries $a_1,...,a_n$ of $a$ in a non-increasing fashion;
\begin{equation*}
a_{\alpha_1} \geq a_{\alpha_2} \geq \cdots \geq a_{\alpha_n}.
\end{equation*}

The \textit{containment ordering}, ``$\leq_c$,'' on $\Z^n$ is then defined by
\begin{equation*}
a=(a_1,...,a_n) \leq_c b=(b_1,...,b_n) \iff a_{\alpha_j} \leq b_{\alpha_j}\ \text{for all}\  j=1,...,n.
\end{equation*}
where $\widetilde{a} = (a_{\alpha_1},....,a_{\alpha_n})$, and $\widetilde{b} = (b_{\alpha_1},....,b_{\alpha_n})$.

Let $k\in \{1,...,n\}$.  The \textit{$k$'th truncation}, $a(k)$ of $a=(a_1,...,a_n)$ is defined to be
\begin{equation*}
a(k)=(a_1, a_2,...,a_k).
\end{equation*}

We represent the elements of the symmetric group $S_n$ by $n$-tuples; for $v \in S_n$ let $(v_1,...,v_n)$ be the sequence where $v_j$ is the row index of the nonzero entry in the $j$'th column of the matrix $v$. For example, the $4$-tuple associated with the permutation matrix
\begin{equation}\label{E:Pexample}
v=\begin{pmatrix}
0 & 1 & 0 & 0 \\
0 & 0 & 0 & 1 \\
1 & 0 & 0 & 0 \\
0 & 0 & 1 & 0
\end{pmatrix}\ \text{is}\ (3142).
\end{equation}
In general, we write $v=(v_1,...,v_n)$ for the corresponding permutation matrix.

\begin{defn}
The Deodhar ordering, $\leq_D$, on $S_n$ is defined by
\begin{equation}\label{D:deodharsdef}
v=(v_1,...,v_n) \leq_D w=(w_1,...,w_n) \iff \widetilde{v(k)} \leq_c \widetilde{w(k)}\ \text{for all}\ k=1,...,n.
\end{equation}
\end{defn}

\begin{rem}
The Deodhar ordering, $\leq_D$ is equivalent to the Bruhat-Chevalley ordering on $S_n$. Although there seems to be no published proof of this fact, it follows as a corollary of our main theorem.
\end{rem}

For the rook monoid $R_n$, a combinatorial description of the Bruhat-Chevalley ordering is given in \cite{PennellPutchaRenner}. We summarize it here.

We represent the elements of  $R_n$ by $n$-tuples of nonnegative integers. Given $x=(x_{ij}) \in R_n$, let $(a_1,...,a_n)$ be the sequence defined by
\begin{equation}\label{E:oneline}
a_j =
\begin{cases}
0,  &\text{if the $j$th column consists of zeros;}\\
i,   &\text{if $x_{ij}=1$.}
\end{cases}
\end{equation}
For example, the sequence associated with the matrix
\begin{equation*}
\begin{pmatrix}
0 & 0 & 0 & 0 \\
0 & 0 & 0 & 0 \\
1 & 0 & 0 & 0 \\
0 & 0 & 1 & 0
\end{pmatrix}
\end{equation*}
is $(3040)$.

\begin{thm}\label{T:PPR}\cite{PennellPutchaRenner}
Let $x = (a_1,...,a_n)$, $y=(b_1,...,b_n) \in R_n$. Then the Bruhat-Chevalley order on
$R_n$ is the smallest partial order on $R_n$ generated by declaring
$x \leq y$ if either
\begin{enumerate}
\item there exists an $1 \leq i \leq n$ such that $b_i> a_i$ and $b_j = a_j$ for all $j\neq i$, or
\item  there exist  $1 \leq i < j \leq n$ such that $b_i=a_j,\ b_j=a_i$ with $b_i > b_j$, and for all $k\notin \{i,j\}$, $b_k = a_k$.
\end{enumerate}
\end{thm}

For example, let $x = (21403)$ and $y= (35201)$ in $R_5$. Then $x \leq_{PPR} y$ because
\begin{eqnarray*}
 (21403) &\leq_{PPR}& (31402)\ \text{by Theorem \ref{T:PPR} part 2}\\
         &\leq_{PPR}& (34102)\ \text{by Theorem \ref{T:PPR} part 2}\\
         &\leq_{PPR}& (35102)\ \text{by Theorem \ref{T:PPR} part 1}\\
         &\leq_{PPR}& (35201)\ \text{by Theorem \ref{T:PPR} part 2}.
\end{eqnarray*}

\begin{rem}
In Proposition 15.23 of \cite{MS05}, Miller and Sturmfels describe the particular case of Theorem \ref{T:PPR} where $y\in S_n$.
\end{rem}

For the sake of notation, the partial ordering defined by the Theorem
\ref{T:PPR} is denoted by ``$\leq_{PPR}$,'' and refered to as the ``Pennell-Putcha-Renner'' ordering on $R_n$.

Notice that Deodhar's ordering (\ref{D:deodharsdef}) on $S_n$ can be defined verbatim on the rook monoid.

\begin{defn} \label{deodonrn.defn}
The {\em Deodhar ordering} $\leq_D$ on $R_n$ is defined as follows.
\begin{equation}\label{D:deodharonrn.def}
v=(v_1,...,v_n) \leq_D w=(w_1,...,w_n) \iff \widetilde{v(k)} \leq_c \widetilde{w(k)}\ \text{for all}\ k=1,...,n.
\end{equation}
\end{defn}

\begin{example}
Let $x=(4,0,2,3,1)$, and let $y=(4,3,0,5,1)$. Then $x \leq_D y$, because
\begin{eqnarray*}
\widetilde{x(1)}=(4) &\leq_c& \widetilde{y(1)}=(4) ,\\
\widetilde{x(2)}=(4,0) &\leq_c& \widetilde{y(2)}=(4,3), \\
\widetilde{x(3)}=(4,2,0) &\leq_c& \widetilde{y(3)}=(4,3,0), \\
\widetilde{x(4)}=(4,3,2,0) &\leq_c& \widetilde{y(4)}=(5,4,3,0), \\
\widetilde{x(5)}=(4,3,2,1,0) &\leq_c& \widetilde{y(5)}=(5,4,3,1,0).
\end{eqnarray*}
\end{example}

The main theorem of this article is that, on $R_n$, the Deodhar ordering and the Pennell-Putcha-Renner ordering are identical.

The organization of the paper is as follows.  In Section \ref{S:lengthfunction}, we study the length function on $R_n$. We show that
\begin{thm} \label{P:length}
Let $x = (a_1,...,a_n)  \in R_n$. Then, the dimension $\ell(x)=\dim(BxB)$ of the  orbit $B x B$, is given by
\begin{equation}
\ell(x)  = (\sum_{i=1}^n a_i^*)  - coinv(x),\ \text{where}\
a_i^* =
\begin{cases}
a_i+n-i,  & \text{if}\  a_i\neq 0, \\
0,   & \text{if}\ a_i=0.
\end{cases}
\end{equation}
\end{thm}
In Section \ref{S:lemmas}, we prove two lemmas, which sharpen the theorem of Pennel, Putcha and Renner.  In Section \ref{S:another}, we find an equivalent description of the Deodhar's ordering.
Finally, in Section \ref{S:final}, we prove that
\begin{thm}
The Deodhar ordering $\leq_D$ on $R_n$ is the same as the Pennell-Putcha-Renner $\leq_{PPR}$ ordering on $R_n$.
\end{thm}

\section{\textbf{the length function.}}\label{S:lengthfunction}

It is well known that the symmetric group $S_n$ is a graded poset, grading given by the length function
\begin{equation}\label{E:lengthsymmetric}
\ell(w)= \dim (B w B)=inv(w)+\dim(B)=inv(w)+{n+1\choose 2},
\end{equation}
where $w \in S_n$, and
\begin{equation}\label{E:invpermutation}
inv(w) = |\{ (i,j):\ 1\leq i < j \leq n,\ w_i>w_j \}|.
\end{equation}

In \cite{Renner86}, it is shown that the rook monoid is a graded poset, with respect to the length function
\begin{equation}
\ell(\sigma)= \dim (B\sigma B),\ \sigma \in R_n.
\end{equation}
In this section we give a combinatorial formula, similar to (\ref{E:lengthsymmetric}), for the length function on $R_n$.

Let $R_n^1$ be the set of all rank one elements of $R_n$. We denote the elements of $R_n^1$ by $E_{ij}=(e_{rs}) \in R_n$, where
\begin{equation*}
e_{rs} =
\begin{cases}
1,  &\text{if $r=i$, and $s=j$,}\\
0,   &\text{otherwise.}
\end{cases}
\end{equation*}
Let $\mathbf{T}_n$ be the set of all upper triangular matrices in $\mathbf{M}_n$.

\begin{lem}\label{L:dim1}
Let $B$ be the Borel subgroup of invertible upper triangular matrices, and let $x=(x_{rs})$ be an element of $R_n$. Then, the dimension $\dim(Bx)$ is equal to the the dimension of the linear subspace $\mathbf{T}_n x$ of $\mathbf{M}_n$, which is spanned by the following set;
\begin{equation*}
\{E_{ij}\in R_n^1:\ there\ exists\ a\ nonzero\  entry\ x_{rs}\ of\ x\ with\ s=j\ and\ r \geq i
\}.
\end{equation*}
\end{lem}

\begin{proof}
The linearity of $\mathbf{T}_n x \subset \mathbf{M}_n$ is clear. Since $\overline{Bx} = \overline{B} x = \mathbf{T}_n x$, and since the geometric dimension of a linear space is the same as its vector space dimension, $\dim (Bx) = \dim (\overline{B x }) = \dim (\mathbf{T}_n x)$. It is easy to see that, $\mathbf{T}_n x$ is spanned by  $R_n^1 \cap \mathbf{T}_n x$. Matrix multiplication shows that $E_{i,j} \in R_n^1 \cap \mathbf{T}_n x$ if and only if there exists a nonzero entry $x_{rs}$ of $x$ with $r\geq i$ and $s=j$.
\end{proof}

\begin{lem}\label{L:dim2}
Let $B$ be the Borel subgroup of invertible upper triangular matrices, and let $x=(x_{rs})$ be an element of $R_n$. Then, the dimension $\dim(xB)$ is equal to the the dimension of the linear subspace $x\mathbf{T}_n$ of $\mathbf{M}_n$, which is spanned by the following set;
\begin{equation*}
\{E_{ij}\in R_n^1:\ there\ exists\ a\ nonzero\  entry\ x_{rs}\ of\ x\ with\ r=i\ and\ s \leq j
\}.
\end{equation*}
\end{lem}

\begin{proof}
Identical to the proof of Lemma \ref{L:dim1}.
\end{proof}

\begin{example}\label{e:4023}
Let $x\in R_4$ be given by the matrix
\begin{equation*}
x= \begin{pmatrix}
0 & 0 & 0 & 0 \\
0 & 0 & 1 & 0 \\
0 & 0 & 0 & 1 \\
1 & 0 & 0 & 0
\end{pmatrix}.
\end{equation*}
Then, a generic element of $\mathbf{T}_4x$ is of the form
\begin{equation*}
\begin{pmatrix}
a_{11} & a_{12} & a_{13} & a_{14} \\
0 & a_{22} & a_{23} & a_{24} \\
0 & 0 & a_{33} & a_{34} \\
0 & 0 & 0 & a_{44}
\end{pmatrix}
\begin{pmatrix}
0 & 0 & 0 & 0 \\
0 & 0 & 1 & 0 \\
0 & 0 & 0 & 1 \\
1 & 0 & 0 & 0
\end{pmatrix}
 = \begin{pmatrix}
a_{14} & 0 & a_{12} & a_{13} \\
a_{24} & 0 & a_{22} & a_{23} \\
a_{34} & 0 & 0 & a_{33} \\
a_{44} & 0 & 0 & 0
\end{pmatrix},
\end{equation*}
for some $a_{ij}\in \F$. Therefore, $\dim (\mathbf{T}_4 x)= 9$.
Similarly, an arbitrary element of $x\mathbf{T}_4$ is of the form
\begin{equation*}
\begin{pmatrix}
0 & 0 & 0 & 0 \\
0 & 0 & 1 & 0 \\
0 & 0 & 0 & 1 \\
1 & 0 & 0 & 0
\end{pmatrix}
\begin{pmatrix}
b_{11} & b_{12} & b_{13} & b_{14} \\
0 & b_{22} & b_{23} & b_{24} \\
0 & 0 & b_{33} & b_{34} \\
0 & 0 & 0 & b_{44}
\end{pmatrix}
 = \begin{pmatrix}
0 & 0 & 0 & 0 \\
0 & 0 & b_{33} & b_{34} \\
0 & 0 & 0 & b_{44} \\
b_{11} & b_{12} & b_{13} & b_{14}
\end{pmatrix},
\end{equation*}
for some $b_{ij}\in \F$.
Thus $\dim (x \mathbf{T}_4)= 7$.
\end{example}

\begin{rem}\label{R:oneline}
Let $x=(a_1,...,a_n)$ be the ``one line" representation for $x=(x_{rs})\in R_n$, as in \ref{E:oneline}.
If $a_i\neq 0$ for some $i\in \{1,...,n\}$, then $a_i$ is the row index of a nonzero entry $x_{a_i i}$ of $x$. Therefore, $E_{r,s} \in R_n^1 \cap \mathbf{T}_n x$ if and only if there exists a nonzero entry of $x$ at the position $(a_i,i)$ with $s=i$ and $r\geq  a_i$. Similarly, $E_{r,s} \in R_n^1 \cap x \mathbf{T}_n$ if and only if there exists a nonzero entry of $x$ at the position $(a_j,j)$ with $r=a_j$ and $s \leq j$.
\end{rem}

\begin{defn}
Let $x=(a_1,....,a_n)\in R_n$.  A pair $(i,j)$ of indices $1\leq i<j \leq n$ is called a \textit{coinversion pair} for $x$, if $0< a_i < a_j$. By abuse of notation, we use
\textit{coinv} for both the set of coinversion pairs of $x$, as well as its cardinality.
\end{defn}

\begin{example}
Let $x=(4,0,2,3)$. Then, the only coinversion pair for $x$ is $(3,4)$. Therefore, $coinv(x)=1$.
\end{example}

\begin{thm} \label{T:length}
Let $x = (a_1,...,a_n)  \in R_n$. Then, the dimension, $\ell(x)=\dim(BxB)$ of the  orbit $B x B$ is given by
\begin{equation}
\ell(x)  = (\sum_{i=1}^n a_i^*)  - coinv(x),\ \text{where}\
a_i^* =
\begin{cases}
a_i+n-i,  & \text{if}\  a_i\neq 0 \\
0,   & \text{if}\ a_i=0
\end{cases}
\end{equation}
\end{thm}

\begin{proof}
Recall from \cite{Renner95} that the dimension of the orbit $BxB$ can be calculated
by
\begin{equation}\label{E:dimformula}
\dim (BxB) = \dim (Bx) + \dim (xB) - \dim (Bx \cap xB).
\end{equation}
By Lemma \ref{L:dim1}, $\dim (Bx)$ is the number of positions on or above some nonzero entry of the matrix $x\in R_n$. In other words, by the Remark \ref{R:oneline}, if $x=(a_1,...,a_n)$, then  $\sum_{i=1}^n a_i$ is equal to $\dim(Br)$.

Similarly, by Lemma \ref{L:dim2}, $\dim (xB)$ is the number of positions on or to the right of some nonzero entry of $x$. The number of positions on and to the right of the nonzero entry at the $(a_i,i)$'th position of the matrix $x$ is equal to $n-i+1$. This shows that
\begin{equation*}
\dim (Bx) + \dim (xB) = \sum_{i=1}^n \overline{a_i},
\end{equation*}
where
\begin{equation*}
\overline{a_i} =
\begin{cases}
a_i+n-i+1,  & \text{if}\  a_i\neq 0, \\
0,   & \text{if}\ a_i=0.
\end{cases}
\end{equation*}
The number of nonzero entries of $x$ is denoted by $rank(x)$.
Thus, we have
\begin{equation*}
\dim (Bx) + \dim (xB) = \sum_{i=1}^n a_i^* + rank(x),
\end{equation*}
where
\begin{equation*}
a_i^* =
\begin{cases}
a_i+n-i,  & \text{if}\  a_i\neq 0, \\
0,   & \text{if}\ a_i=0.
\end{cases}
\end{equation*}

Therefore, it is enough to prove that
\begin{equation*}
\dim (Bx \cap xB) = rank(x) + coinv((a_1,....,a_n)).
\end{equation*}

By a similar argument as in the proof of Lemma \ref{L:dim1}, the dimension of $Bx \cap x B$ is equal to $\dim( \mathbf{T}_n x \cap x \mathbf{T}_n)$, which is equal to the cardinality of the set $R_n^1 \cap \mathbf{T}_n x \cap x \mathbf{T}_n$.

Let $E_{rs}\in R_n^1 \cap \mathbf{T}_n x \cap x \mathbf{T}_n$ be a rank 1 element whose nonzero entry is at the $(r,s)$'th position. By the Remark \ref{R:oneline}, $E_{rs} \in R_n^1 \cap \mathbf{T}_n x \cap x \mathbf{T}_n$ if and only if there exist nonzero entries of $x$ at the positions $(a_i,i)$ and $(a_j,j)$ such that $r\geq a_i,\ s=i$ and $r=a_j,\ s \leq j$. We have two possibilities. Either $(a_i,i)=(a_j,j)$, or not.
Clearly, the number of times that the equality $(a_i,i)=(a_j,j)$ holds true is equal to the $rank(x)$. On the other hand, if $(a_i,i)\neq (a_j,j)$, then we see that $i < j$ and $0<a_i<a_j$. Therefore, the number of times that $(a_i,i)\neq (a_j,j)$, is equal to the number of coinversions of the sequence $(a_1,...,a_n)$. Therefore,
\begin{equation*}
\dim (Bx \cap xB) = |R_n^1 \cap \mathbf{T}_n x \cap x \mathbf{T}_n| =rank(x) + coinv((a_1,....,a_n)).
\end{equation*}
\end{proof}

\begin{rem}
Let $x=(a_1,...,a_n) \in R_n$ be a permutation. Then
\begin{eqnarray*}
\ell(x) &=& (\sum_{i=1}^n a_i+n-i) - coinv(x)\\
&=&{n+1 \choose 2}+{n\choose 2} - coinv(x)\\
&=& {n+1 \choose 2}+inv(x),\\
\end{eqnarray*}
which agrees with the formula (\ref{E:lengthsymmetric}).
\end{rem}

\begin{example}
We continue with the notation of the example \ref{e:4023}. The generic element of $\mathbf{T}_4x \cap x\mathbf{T}_4$ has the form
\begin{equation*}
\begin{pmatrix}
0 & 0 & 0 & 0 \\
0 & 0 & * & * \\
0 & 0 & 0 & * \\
*  &  0  &  0 & 0
\end{pmatrix},
\end{equation*}
where $*$ denotes an arbitrary element of $\F$. Therefore, $\dim(\mathbf{T}_4 x \cap x \mathbf{T}_4) = 4$, and by formula \ref{E:dimformula}, we have
$\dim(BxB) = 9+7 - 4 = 12$. On the other hand, $x$ is represented in ``one line" notation by $(4,0,2,3)$, and by Theorem \ref{P:length} we have
\begin{equation*}
\ell(x) = (4+4-1)+(2+4-3)+(3+4-4)-1=12.
\end{equation*}
\end{example}

\section{\textbf{Two important lemmas.}}\label{S:lemmas}

Recall that we denote the Bruhat-Chevalley ordering on $R_n$, as in Theorem \ref{T:PPR}, by $\leq_{PPR}$. The following two lemmas are critical for deciding if $x \leq_{PPR} y$ is a covering relation.

\begin{lem} \label{L:PPRcovering0}
Let $x=(a_1,...,a_n)$ and $y = (b_1,...,b_n)$ be elements of $R_n$. Suppose that $a_k=b_k$ for all $k =\{ 1,...,\widehat{i},...,n\}$ and $a_i < b_i$. Then, $\ell(y) = \ell( x)+ 1$ if and only if either
\begin{enumerate}
\item $b_i = a_i +1$, or
\item there exists a sequence of indices $1 \leq j_1< \cdots < j_s < i$ such that the set $\{a_{j_1},...,a_{j_s}\}$ is equal to $\{a_i+1,...,a_i+s\}$, and $b_i=a_i+s+1$.
\end{enumerate}
\end{lem}
\begin{proof}
Note that by the hypotheses of the lemma, Theorem \ref{T:PPR} implies that $ x \leq_{PPR} y$.
We first show that if (1) or (2) holds, then $\ell(y)= \ell(x)+1$, in other words $y$ covers $x$.

If $b_i = a_i +1$, then by the Theorem \ref{T:length} the lemma follows. So, we assume that there exists a sequence of indices $1 \leq j_1< \cdots < j_s < i$ such that the set $\{a_{j_1},...,a_{j_s}\}$ is equal to $\{a_i+1,...,a_i+s\}$, and $b_i=a_i+s+1$.  Then,
\begin{eqnarray*}
\ell(y) &=&  \sum_{j=1}^n b_j^* - coinv(y)\\
   &=&  ( \sum_{j=1, j\neq i}^n a_j^* )+ b_i^*  - coinv(y)\\
   &=&  ( \sum_{j=1, j\neq i}^n a_j^* )+ a_i+s+1+n-i  - coinv(y)\\
   &=& ( \sum_{j=1}^n a_j^* )+ s+1 - coinv(y).
\end{eqnarray*}

Now it suffices to show that $coinv(y)=s+coinv(x)$. Observe that, when we replace $a_i$ by $b_i$, the following set of pairs, which are not coinversion pairs for $x$,
\begin{equation*}
\{ (j_k, i)|\ k=1,....,s\},
\end{equation*}
become coinversion pairs for $y$.
Also, upon replacing the entry $a_i$ by $b_i$, a coinversion pair of $x$ of the form $(l,i)$ or $(i,l)$ (where $l\neq j_k$) stays to be a coinversion pair for $y$. Therefore,
\begin{equation*}
coinv(y) = s+ coinv(x),
\end{equation*}
and hence $\ell(y) = \ell(x) +1$.

We proceed to prove the converse statement. Assume that $\ell(y)=\ell(x)+1$.  Since $b_i > a_i$, there exists $d>0$ such that $b_i=a_i+d$. Without loss of generality we may assume that $d>1$. Then the length of $y$ can be computed as follows.
\begin{eqnarray*}
\ell(y) &=&  \sum_{j=1}^n b_j^* - coinv(y)\\
   &=&  ( \sum_{j=1, j\neq i}^n a_j^* )+ b_i^*  - coinv(y)\\
   &=&  ( \sum_{j=1, j\neq i}^n a_j^* )+ a_i+d+n-i  - coinv(y)\\
   &=& ( \sum_{j=1}^n a_j^* )+ d - coinv(y)\\
   &=& \ell(x)+d+coinv(x)-coinv(y).
\end{eqnarray*}
Hence $d+coinv(x)-coinv(y)=1$, or $coinv(y)-coinv(x)=d-1$.
We inspect the difference $coinv(x)-coinv(y)$ more closely.
If $(k,i)$ with $k<i$ is a coinversion for $x$, then it stays to be a coinversion for $y$, as well.
Clearly this is also true for the pairs of the form $(k,l)$ where $k<i<l$, or $i<k<l$, or $k<l<i$.

Therefore, the difference between $coinv(y)$ and $coinv(x)$  occurs at the pairs of the form
\begin{enumerate}
\item $(k,i),\ k<i$ such that $a_i < a_k < b_i$, or
\item $(i,l),\ i<l$, such that $a_i < a_l < b_i$.
\end{enumerate}
In the first case, some new coinversions are added, and in the second case some coinversions are deleted.
Let us call the number of pairs of the first type by $n_1$ and the number of pairs of the second type by $n_2$. Then, $coinv(y) = coinv(x) + n_1-n_2$, or $coinv(y)-coinv(x) = n_1-n_2$. Obviously $0 \leq n_1, n_2 \leq d-1$ (because $b_i=a_i+d$). Hence, we have that $n_1=d-1$, and that $n_2=0$. Therefore, the following is true: any $a_k$ between $a_i$ and $a_i+d=b_i$ appears before the $i$'th position. This completes the proof.
\end{proof}

\begin{example}
Let $x=(4,0,5,0,3,1)$, and let $y=(4,0,5,0,6,1)$. Then $\ell(x)= 21$, and $\ell(y)=22$. Let $z=(6,0,5,0,3,1)$. Then $\ell(z)=23$.
\end{example}

\begin{lem} \label{L:PPRcovering1}
Let $x=(a_1,...,a_n)$ and $y = (b_1,...,b_n)$ be two elements of $R_n$. Suppose that $a_j=b_i,\ a_i = b_j$ and $b_j < b_i$ where $i < j$. Furthermore, suppose that for all $k\in \{1,...\widehat{i},...,\widehat{j},...,n\}$, $a_k=b_k$.
Then, $\ell(y) = \ell( x)+ 1$ if and only if for $s=i+1,...,j-1$, either $a_j< a_s$, or $a_s < a_i$.
\end{lem}
\begin{proof}
Suppose that $x$ and $y$ are as in the hypothesis. Also suppose also that $\ell(y) = \ell(x)+1$.  We proceed to show that for $s=i+1,...,j-1$, either $a_j< a_s$, or $a_s < a_i$. Clearly, the sets $\{a_1,...,a_n\}$ and $\{b_1,...,b_n\}$ are equal, hence $\sum_{t=1}^n a_t  = \sum_{t=1}^n b_t $. Therefore, the difference between $\ell(x)$ and $\ell(y)$ is determined by the associated coinversion sets of $x$ and $y$.

Assume that there exists an $s \in \{i+1,...,j-1\}$ such that $a_i < a_s < a_j $. Then, upon interchanging $a_i$ with $a_j$ to get $y$ from $x$, the pairs $(i,s),\ (s,j)$ and $(i,j)$ are no longer coinversions for $y$. This shows that for every $s=i+1,...,j-2$ with $a_i < a_s < a_j$, we obtain that $\ell(y)\geq\ell(x)+2$. This contradicts the assumption that $\ell(y) = \ell(x)+1$. Therefore, there exists no $s \in \{i+1,...,j-1\}$ such that $a_i < a_s < a_j $.

Conversely, assume that for every $s=i+1,...,j-1$, we have $a_i > a_s$ or $a_s > a_j$.  If $a_i > a_s$, then, the pair $(s,j)$ is a coinversion pair for both $x$ and $y$. On the other hand, the pair $(i,s)$ is neither a coinversion for $x$ nor for $y$. Similarly, if $(a_s>a_j)$, then the pair $(i,s)$ is a coinversion pair for both $x$ and $y$. Also, the pair $(s,j)$ is not a coinversion pair for $x$ and neither for $y$. Therefore, we conclude that at any pair of the form $(k,l)$ with $i \leq k < l \leq j$, the coinversion is not affected. It remains to check  pairs of the form $(k,l)$ with either $k <i$, or $j< k$. In the first case, i.e.,  $k<i$, as $a_i$ is interchanged with $a_j$, the contribution of $(k,l)$ to the coinversion situation does not change,  since relative positions of $a_k$ and $a_l$ do not alter. Similarly, in the second case, i.e., $j<k$, since the relative positions of $a_k$ and $a_l$ do not alter, their contribution to coinversion do not change. Therefore, the only coinversion change occurs at the pair $(i,j)$,  and hence, $\ell(y)= \ell(x)+1$.
This completes the proof.
\end{proof}

\begin{example}
Let $x=(2,6,5,0,4,1,7)$, and let $y=(4,6,5,0,2,1,7)$. Then $\ell(x)= 35$, and $\ell(y)=36$. Let $z=(7,6,5,0,4,1,2)$. Then $\ell(z)=42$.
\end{example}

\section{\textbf{Another characterization of $\leq_D$.}}\label{S:another}

As mentioned in the introduction, our goal is to show that
the $\leq_D$ ordering on $R_n$ is the same as to the $\leq_{PPR}$ ordering.  In this section, we find another, useful characterization of the Deodhar ordering.

%For the convenience of the reader, we start by recalling the definition of the Deodhar ordering.

% Let $x=(a_1,...,a_k)$, and  $y=(b_1,...,b_k)$ be $n$-tuples of nonnegative integers. We denote by
%$\widetilde{x} = (a_{\alpha_1},....,a_{\alpha_k})$ (resp. by $\widetilde{y}$) the reordering of the entries of
%$x$ (resp. of $y$) in such a way that $a_{\alpha_1} \geq a_{\alpha_2} \geq \cdots \geq a_{\alpha_k}$
%(resp.  $b_{\alpha_1} \geq \cdots \geq b_{\alpha_k}$).
%The containment ordering is then defined by
%\begin{equation*}
%x \leq_{c} y \iff a_{\alpha_j} \leq b_{\alpha_j} \text{for all}\ j=1,...,k.
%\end{equation*}

%Let $a(k)=(a_1,...,a_k)$ be the $k$'th truncation of $a=(a_1,...,a_n)$. Then,

%\begin{equation*}
%x=(a_1,...,a_n) \leq_D y=(b_1,...,b_n) \iff \widetilde{x(k)} \leq_c \widetilde{y(k)},\ k=1,...,n.
%\end{equation*}

%Let $x=(x_{ij})\in R_n$ be given in the matrix form. The rank
%matrix $\Gamma(v)=(\Gamma_{ij}(x))$ of $x$ is, as defined before
%\begin{equation}
%\Gamma_{ij}= |\{x_{kl}|\ x_{kl}\neq 0, k\leq i,\ l\leq j \}|.
%\end{equation}

%\begin{defn}
%Let $x=(x_{ij})$ and $y=(y_{ij})$ be two permutation matrices of
%size $n$. We write $x \leq_P y $, and say that $x$ is less than or
%equal to $y$ in the Proctor ordering if and only if
%\begin{equation}
%\Gamma_{ij}(x)\leq \Gamma _{ij}(y)\ \text{for every}\ 1\leq i, j
%\leq n.
%\end{equation}
%\end{defn}

%We proceed to prove that ``$\leq_D$'' is equal to  ``$\leq_P$''
%on $R_n$.

\begin{defn}
Let $x=(a_1,....,a_n) \in R_n$, and let $r\in \{1,...,n\}$, and finally let $a \in \Z$. We define
\begin{equation*}
%\gamma(x,a,r) = \{ a_i\in x |\ a_i < a,\ i>r\}.
\Gamma(x,a) = \{ a_i\in x|\ a_i > a\}.
\end{equation*}
\end{defn}
%\begin{rem}
%It fairly easy to see that $\Gamma_{ij}(x) = |\Gamma(x(i),a_j)|$.
%In other words, the rank matrix $\Gamma(x)$ is given by
%\begin{equation*}
%\Gamma(x) =
%(|\Gamma(x(i),a_j)|).
%\end{equation*}
%\end{rem}

\begin{rem}\label{R:cposition}
Let $a_i$ be a nonzero entry of $x=(a_1,....,a_n)\in R_n$. Then,
$|\Gamma(x,a_i)|+1$ is the position of $a_i$ in the reordering
$\widetilde{x} = (a_{\alpha_1} \geq \cdots \geq a_{\alpha_n})$ of the entries of $x$. For example,
if $x=(3,0,5,1,0,4)$, then $\widetilde{x}=(5,4,3,1,0,0)$, and $|\Gamma(x,1)|+1=4$.
\end{rem}

\begin{prop}\label{P:containment}
Let $x=(a_1,....,a_n)$ and $y=(b_1,...,b_n)$ be two elements from $R_n$. Then $x \leq_c y$ if and only if $|\Gamma(x,a_k)| \leq |\Gamma(y,a_k)|$ for all $k=1,....,n.$
\end{prop}

\begin{proof}

Let $\widetilde{y}=  (b_{\alpha_1} \geq \cdots \geq b_{\alpha_n})$ and $\widetilde{x} = (a_{\alpha_1} \geq \cdots \geq a_{\alpha_n})$ be the reorderings of the entries of $y$ and of $x$ respectilvely. Then, by the Remark \ref{R:cposition}, $a_{\alpha_{s+1}}$ is the entry $a_k$ of $x$ for which $|\Gamma(x,a_k)|=s$. Therefore, $b_{\alpha_{s+1}} \geq a_{\alpha_{s+1}}$ if and only if the number of entries of $y$ which are larger than $a_k$ is more than the number of entries of $x$ which are larger than $a_k$. In other words,  $b_{\alpha_{s+1}} \geq a_{\alpha_{s+1}}$ if and only if $|\Gamma(x,a_k)| \leq |\Gamma(y,a_k)|$. Thus $x \leq_c y$ if and only if $|\Gamma(x,a_k)| \leq |\Gamma(y,a_k)|$, for all $k=1,....,n.$
\end{proof}

As a corollary of the Proposition \ref{P:containment}, we have

\begin{cor}\label{C:containment}
Let $x=(a_1,....,a_n),$ and $y=(b_1,...,b_n)$ be two elements of $R_n$. Then $y \geq_D x$ if and only if for all $1\leq k \leq n$ and for all $m \leq k$,  $|\Gamma(x(k),a_m)| \leq |\Gamma(y(k),a_m)|$.
\end{cor}
\begin{proof}
Immediate from Proposition \ref{P:containment}, and the definition of the Deodhar ordering.
\end{proof}

\begin{example}
Let $x=(a_1,a_2,a_3)=(1,0,3)$ and let $y=(b_1,b_2,b_3)=(3,0,2)$. Then
\begin{eqnarray*}
|\Gamma(x(1),a_1)|=0 &\leq& |\Gamma(y(1),a_1)|=1,\\
|\Gamma(x(2),a_1)|=0 &\leq& |\Gamma(y(2),a_1)|=1,\\
|\Gamma(x(2),a_2)|=1 &\leq& |\Gamma(y(2),a_2)|=2,\\
|\Gamma(x(3),a_1)|=1 &\leq& |\Gamma(y(3),a_1)|=2,\\
|\Gamma(x(3),a_2)|=2 &\leq& |\Gamma(y(3),a_2)|=2,\\
|\Gamma(x(3),a_3)|=0 &\leq& |\Gamma(y(3),a_3)|=0.\\
\end{eqnarray*}
Therefore, $x\leq_D y$.
\end{example}

\begin{rem}\label{R:deodhar0}
It follows from the definition of the Deodhar ordering that if $(a_1,....,a_n) \leq_D (b_1,....,b_n)$, then
$(a_1,...,a_k) \leq_D (b_1,...,b_k)$ for any $k\in \{1,....,n\}$.  Also, by repeated application of Proposition \ref{P:containment}, it follows that  $$(a_1,....,a_k,c_{k+1},...,c_m) \leq_D (b_1,....,b_k,c_{k+1},...,c_m)$$ for any set $\{c_{k+1},....,c_m\}$ of nonnegative integers.
\end{rem}

%\begin{rem} (Not needed)
%Recall that the length of an element $x=(a_1,....,a_n) \in R_n$ is
%defined as $\sum_{i=1}^n a_i^* - coinv(x)$, where $a_i^*= a_i + n-i$% if $a_i \neq 0$, and $a_i^*=0$ if $a_i=0$. Clearly, this
%definition of length still works for any string of nonnegative
%integers.
%\end{rem}

\section{\textbf{The Main Theorem.}}\label{S:final}

 We show in this section that the covering relation for the ordering $\leq_{PPR}$ on $R_n$ is  the same as the covering relation for the ordering $\leq_{D}$ on $R_n$. Our notation for these covering relations is
``$y \rightarrow_D x$,'' and ``$y \rightarrow_{PPR} x$,'' respectively.

\begin{lem}\label{L:preparation2}
Let $x=(a_1,....,a_n),\ y= (b_1,...,b_n)$ and $z=(c_1,...,c_n)$ be three elements from $R_n$ such that $a_k = b_k$ for all $k\in \{1,...,\widehat{i},...,n\}$ and $a_i < b_i$. Furthermore, suppose that    $c_k=a_k$ for $k=1,...,i$. If $x \leq_D z \leq_D y$ and $\ell(y) = \ell(x)+1$, then  $z=x$.
\end{lem}
\begin{proof}
Assume otherwise that $z \neq x$.  Let $j > i$ be the smallest number such that $c_k=a_k$ for $k<j$ but $c_j  \neq a_j$.  Since $ x \leq_D z$, it cannot be true that $c_j < a_j$.
So, we have that $a_j < c_j$. This, in particular, implies that $c_j$ is nonzero.

 We now compare $c_j$ with $a_i$. Observe that $c_j=a_i$ is not possible. Thus, there are two cases; either $c_j < a_i$ or $a_i < c_j$.

 We proceed with the first case. Then, we have $a_j=b_j  < c_j < a_i = c_i <  b_i$.  Recall that $\Gamma(z(j),b_j) = \{ c_k|\ b_j < c_k,\  k=1,...,j\}$, and  that $\Gamma(y(j),b_j) = \{ b_k|\ b_j < b_k,\ k=1,...,j\}$.

Since,
\begin{equation*}
\{ b_1,..., b_j\} \setminus \{b_i,b_j\} = \{ c_1,...,c_j \} \setminus \{c_j,c_i\}.
\end{equation*}
and since $b_j < c_j<  c_i$, we see that  $|\Gamma(z(j),b_j) | = |\Gamma(y(j),b_j)| +1$. By the Remark \ref{R:cposition}, this is equal to the position of $b_j$ in $\widetilde{y(j)}$. In other words, the position of $b_j$ in $\widetilde{y(j)}$ is $\alpha_{s}=|\Gamma(z(j),b_j) |$.

On the other hand, $|\Gamma(z(j),b_j) |$ is equal to the number of entries of $z(j)$ which are larger than $b_j$. Therefore, in $c_{\alpha_s}>b_{\alpha_s}=b_j$,  But this is a contradiction to $z(j) \leq_c y(j)$.
Therefore, the first case, $c_j<a_i$ is not possible.

We assume that $a_i < c_j$. Since $a_j=b_j$, and since by our initial assumption $a_j < c_j$, we have that $b_j< c_j$. Since $i<j$, and since $\ell(y) = \ell(x) +1$, Lemma \ref{L:PPRcovering0} implies  that $b_i \leq c_j$.

Assume for a second that $b_i < c_j$.  Let $\alpha_s$ be the position of $c_j$ in $\widetilde{z(j)}$.
Since,
\begin{equation*}
\{ b_1,..., b_j\} \setminus \{b_i,b_j\} = \{ c_1,...,c_j \} \setminus \{c_j,c_i\},
\end{equation*}
and since, $c_i<c_j$, $b_i< c_j$, and $b_j < c_j$,
we see that  $|\Gamma(z(j),c_j) | = |\Gamma(y(j),c_j)|$. Therefore, $b_{\alpha_s}<c_{\alpha_s}=c_i$. But this contradicts the fact that $z(j) \leq_c y(j)$.

Therefore, we assume that $b_i=c_j$.  Since $b_j=a_j<c_j=b_i$, and since $\ell(y) = \ell(x) +1$, Lemma \ref{L:PPRcovering0} implies  that $b_j \leq c_i=a_i < c_j$. We look at the position $\alpha_s$ of $c_i$ in $\widetilde{z(j)}$. Since,
\begin{equation*}
\{ b_1,..., b_j\} \setminus \{b_i,b_j\} = \{ c_1,...,c_j \} \setminus \{c_j,c_i\},
\end{equation*}
we see that  $|\Gamma(z(j),c_i) | = |\Gamma(y(j),c_i)|$. Therefore, $b_{\alpha_s}<c_{\alpha_s}=c_i$. This contradicts the fact that $z(j) \leq_c y(j)$. We have handled all the cases, and the proof is complete.
\end{proof}

\begin{lem}\label{L:preparation3}
Let $x=(a_1,....,a_n),\ y= (b_1,...,b_n)$ and $z=(c_1,...,c_n)$ be three elements from $R_n$ such that $a_k = b_k$ for all $k\in \{1,...,\widehat{i},...,n\}$ and $a_i < b_i$. Furthermore, $c_k=b_k$ for $k=1,...,i$. If $x \leq_D z \leq_D y$ and $\ell(y) = \ell(x)+1$, then  $z=y$.
\end{lem}

\begin{proof}
We proceed as in the proof of Lemma \ref{L:preparation2}. Assume otherwise that $z \neq y$, and let $j > i$ be the first position where $z$ differs from $y$. Hence, there are now two subcases; either $c_j < b_j$ or else $b_j < c_j$.

In the second case, with $b_j< c_j$, we see that $y(j) <_c z(j)$, which contradicts the fact that $ z \leq_D y$.

Therefore, we assume that $c_j < b_j = a_j$. There are now two subcases; either $c_j < a_i$, or else $a_i < c_j$. We first treat the case $c_j<a_i$.

Recall that $\Gamma(z(j),c_j) = \{ c_k|\ c_j < c_k,\  k=1,...,j\}$, and that $\Gamma(x(j),c_j) = \{ a_k|\ c_j < a_k,\ k=1,...,j\}$. Then, since
\begin{equation*}
\{ a_1,..., a_j\} \setminus \{a_i,a_j\} = \{ c_1,...,c_j \} \setminus \{c_j,c_i\},
\end{equation*}
and $c_j < a_i,\  a_j$, we see that  $|\Gamma(z(j),c_j) | +1 = |\Gamma(x(j),c_j)|$.
This shows the following; if the position of $c_j$ in $\widetilde{z(j)}$ is $\alpha_{s}$, then $a_{\alpha_{s}} > c_{\alpha_{s}} = c_j$, a contradiction to $x(j) \leq_c z(j)$.

We proceed with the case that $a_i < c_j$. Since $\ell(y) = \ell(x) +1$, and $z(j-1) = y(j-1)$, we see that $c_j$ must be larger than $c_i=b_i = a_i +s +1$ (or larger than $c_i=b_i =a_i +1$). Therefore, similar to the above, since
\begin{equation*}
\{ a_1,..., a_n\} \setminus \{a_i,a_j\} = \{ c_1,...,c_n \} \setminus \{c_j,c_i\},
\end{equation*}
and $a_i < c_j <  a_j$, and $c_i < c_j$, we see that  $|\Gamma(z(j),c_j) | +1 = |\Gamma(x(j),c_j)|$. This shows the following; if the position of $c_j$ in $\widetilde{z(j)}$ is $\alpha_{s}$, then $a_{\alpha_{s}} > c_{\alpha_{s}} = c_j$, a contradiction to $x(j) \leq_c z(j)$.

Therefore, we conclude that $z = y$.
\end{proof}

\begin{lem}\label{L:D1}
Let $x=(a_1,....,a_n)$ and $z=(c_1,...,c_n)$ be two elements from $R_n$. Suppose that
$c_i=a_r$ and $c_r=a_i$, with $i<r$. Furthermore, suppose that
$c_k=a_k$, for $k\notin \{i,r\}$.
If $a_r > a_i$, then $z \gneq_D x$.
\end{lem}
\begin{proof}
This follows directly from Corollary \ref{C:containment}.
\end{proof}

\begin{prop}\label{P:Deodharcovering0}
Let $x=(a_1,....,a_n)$ and $y= (b_1,...,b_n)$ be two two elements from $R_n$ such that
$a_k = b_k$ for all $k\in \{1,...,\widehat{i},...,n\}$ and $a_i < b_i$. Then $\ell(y) = \ell(x) +1$ if and only if $y \rightarrow_D x$.

\end{prop}
\begin{proof}
It is clear from the hypotheses that $x<_{PPR} y$, and that $x <_D y$.  We first show that  if $\ell(y) = \ell(x)+1$, then $y \rightarrow_D x$.  Let $z=(c_1,...,c_n) \in R_n$ be such that $x \leq_D z \leq_D y$. Then, since $a_k = b_k$ for $k=1,...,i-1$, we must have $c_k=a_k$, for $k=1,...,i-1$.  In other words, $x(k)=z(k)=y(k)$ for $k=1,...,i-1$. Since $x(i) \leq_c z(i) \leq_c y(i)$, we must also have $a_i \leq c_i \leq b_i$. Therefore, either $a_i=c_i$, or $a_i< c_i$. In the former case, by the Lemma \ref{L:preparation2}, $z$ is identically equal to $x$. Therefore, we have $a_i < c_i \leq b_i$, so that $x<_D z \leq_D y$. We are going to show that  $z=y$.

As in the notation of Lemma \ref{L:PPRcovering0}, if $b_i = a_i+s+1$ for some $s\geq 0$, then we must have $c_i=b_i$. This is because, $c_i$ cannot be strictly larger than $b_i$ (otherwise $z(i) > y(i)$ ), and $c_i$ cannot less than $b_i$ (otherwise $c_i$ has to be one of $\{a_{j_1},...,a_{j_s}\}$, which contradicts with the fact that $z(k)=y(k)$ for all $k=1,....,i-1$). Therefore, $c_k = b_k$ for $k=1,...,i$. By the Lemma \ref{L:preparation3}, we see that $z=y$. Therefore, $\ell(y) = \ell(x) +1$ implies that $y \rightarrow_D x$.

Conversely, assume that $y \rightarrow_D x$. If $b_i=a_i+1$, then it is clear that $\ell(y)=\ell(x)+1$. So, we assume that $b_i=a_i +s+1$, for some $s>0$. To finish the proof, by the Lemma \ref{L:PPRcovering0}, it is enough to show that there exists a sequence of indices $1 \leq j_1< \cdots < j_s < i$ such that $\{a_{j_1},...,a_{j_s}\} = \{a_i+1,...,a_i+s\}$, and $b_i=a_i+s+1$.

 Let $d$ be a number such that $1\leq d \leq s$. If $a_i+d$ does not appear in $y$, then we define $z=(c_1,...,c_n) \in R_n$ to be the sequence such that $c_k = a_k$ for $k\in \{1,....,\widehat{i},...,n\}$ and $c_i = a_i+d$. It is clear that $x \lneq_D z \lneq_D y$. But this contradicts with the hypotheses that $y \rightarrow_D x$. Therefore, the number $a_i+d$ is an entry of $y$.  Assume for a second that $a_i+d=b_t=a_t$ for some $t > i$. Then we define $z=(c_1,...,c_n) \in R_n$ to be the element such that $c_k = a_k$ for $k\in \{1,....,\widehat{i},...,\widehat{t},...,n\}$ and $c_i = a_i+d,\  c_t=a_i$. Then, using the Lemma \ref{L:D1}, it is easy to check that $x \lneq_D z \lneq_D y$, which is a contradiction. Therefore, $t<i$. In other words, for any $1\leq d < s$, the number $a_i+d$ is an entry of $x$, with the index $<i$. This shows that there exists a sequence of indices $1 \leq j_1< \cdots < j_s < i$ such that the set $\{a_{j_1},...,a_{j_s}\}$ is equal to $\{a_i+1,...,a_i+s\}$, and $b_i=a_i+s+1$.
 \end{proof}

\begin{lem}\label{L:preparation4}
Let $x=(a_1,...,a_n),\ y = (b_1,...,b_n)$ and $z=(c_1,..,c_n)$ be three element of $R_n$, such that
$\widetilde{x} = \widetilde{y}$.  If $x \leq_D z \leq_D y$, then $\widetilde{z}= \widetilde{x} = \widetilde{y}$.
\end{lem}
\begin{proof}
By definition of the Deodhar ordering,  $x \leq_D z \leq_D y$ is true if and only if $x(k) \leq_c z(k) \leq_c y(k)$, for all $k=1,....,n$. Recall that $\widetilde{z}$ stands for the reordering, from the largest to smallest entries of $z$.  Therefore, if $\widetilde{z} \neq \widetilde{x}$, then  there exits $1 \leq \alpha_r \leq n$ such that
$a_{\alpha_r} < c_{\alpha_r}$. But since $z(n) \leq_c y(n)$, we see that $c_{\alpha_r} \leq b_{\alpha_r} = a_{\alpha_r}$, a contradiction. Therefore $\widetilde{z}= \widetilde{x}$.
\end{proof}

\begin{lem}\label{L:preparation5}
Let $x=(a_1,....,a_n),\ y= (b_1,...,b_n)$ and $z=(c_1,...,c_n)$ be three elements from $R_n$ such that $\widetilde{x(n-1)}=\widetilde{y(n-1)}=\widetilde{z(n-1)}$, $a_n = b_n$ and $x \leq_D z \leq_D y$. Then, $c_n=a_n = b_n$.
\end{lem}
\begin{proof}
Since $\widetilde{x(n-1)}=\widetilde{y(n-1)}$, and since $a_n = b_n$, we see, by the Lemma \ref{L:preparation4}, that $\widetilde{z}=\widetilde{x}=\widetilde{y}$. This, together with the fact that $\widetilde{z(n-1)}=\widetilde{x(n-1)}=\widetilde{y(n-1)}$, forces the equality $c_n=a_n=b_n$.
\end{proof}

\begin{prop} \label{P:Deodharcovering1}
Let $x=(a_1,...,a_n)$ and $y = (b_1,...,b_n)$ be two elements of $R_n$. Suppose that for some $1 \leq i < j \leq n$, $a_j=b_i,\ a_i = b_j$ and $b_j < b_i$, and $a_k=b_k$ for all $k\in \{1,...\widehat{i},...,\widehat{j},...,n\}$. Then, $\ell(y) = \ell(x)+ 1$ if and only if $y \rightarrow_D x$.
\end{prop}

\begin{proof}
It is clear from Lemma \ref{L:D1} that $x <_D y$. Also, we know from Lemma \ref{L:PPRcovering1} that  $\ell(y) = \ell(x)+ 1$ if and only if for each $s\in \{i+1,...,j-1\}$, either $a_j< a_s$, or $a_s < a_i$.  Throughout the
proof, we shall make use of this.

Suppose first that $y \rightarrow_D x$. Assume that there exists $s \in \{i+1,...,j-1\}$ such that $a_i < a_s < a_j$. Then, define $z =(c_1,...,c_n) \in R_n$ such that $c_k=a_k$ for all $k\notin \{s,j\}$, and,  $c_s = a_j$, $c_j= a_s$. Then, by the repeated applications of Lemma \ref{L:D1}, it is easy to see that $x \lneq_D z \lneq_D y$. But this implies that $y$ does not cover $x$ in the Deodhar ordering, which is a contradiction.  Therefore, $\ell(y) = \ell(x)+1$.

Conversely, suppose that $\ell(y)=\ell(x)+1$. There are two cases; $j=i+1$, or $j> i+1$.
Suppose first that $j=i+1$. Notice that by the Lemma \ref{L:preparation4}, the set of the entries of $z$ is equal to the set of entries of $x$, which is also equal to the set of entries of $y$. Clearly, for $k=1,....,i-1$, we have that $x(k)=z(k)=y(k)$. Since $j=i+1$, we see that $\widetilde{x(j)}=\widetilde{y(j)}$. Thus, by Lemma \ref{L:preparation4}, we see that $\widetilde{z(j)}=\widetilde{x(j)}=\widetilde{y(j)}$. This shows that
either $c_i=a_i$ and $c_j=a_j$, or $c_i=b_i$ and $c_j=b_j$. Finally, for $k>j$, Lemma \ref{L:preparation5} shows that $c_k=a_k=b_k$. Therefore, we conclude, in the case of $j=i+1$, that either $z=x$, or $z=y$.

We proceed with the case that $j>i+1$. By Lemma \ref{L:PPRcovering1},  we know that for $s=i+1,...,j-1$, either $a_j< a_s$, or $a_s < a_i$. Let $z=(c_1,...,c_n)\in R_n$ be such that $x \leq_D z \leq_D y$. Notice that by Lemma \ref{L:preparation4}, the set of the entries of $z$ is equal to the set of entries of $x$. Furthermore,  for $k=1,....,i-1$, we have that $x(k)=z(k)=y(k)$.
Also, since $x(i) \leq_c z(i) \leq_c y(i)$, we must have $a_i \leq c_i \leq b_i$.
We proceed to show that for $s=i+1,...,j-1,j+1,...,n$, $c_s=a_s=b_s$. Once we show this, the proof is finished as follows. By Lemma \ref{L:preparation4}, we know that $\widetilde{z}=\widetilde{x}=\widetilde{y}$. Since $c_s=a_s=b_s$ for all $s\in \{1,...,\widehat{i},...,\widehat{j},...,n \}$, we either have
$c_i=a_i$ and $c_j=a_j$, or $c_i=b_i$ and $c_j=b_j$, in other words, either $z=x$, or $z=y$.

We start by showing that $c_{i+1}=a_{i+1}=b_{i+}$.  By Lemma \ref{L:PPRcovering1},
we know that one of the following is true.

\textit{Case 1.}  $b_{i+1}= a_{i+1} < a_i$, or

\textit{Case 2.}  $b_{i+1}=a_{i+1} > b_i= a_j$.

We start with the first case that $a_{i+1}< a_i \leq c_i$, and we look at the following two subcases: $c_{i+1} < a_{i+1}$ or $c_{i+1} > a_{i+1}$.

\textit{Case 1.1.}  $c_{i+1} < a_{i+1}=b_{i+1}$, or

\textit{Case 1.2}  $c_{i+1} > a_{i+1}=b_{i+1}$.

We first deal with the \textit{Case 1.1.}. Let $\Gamma(x(i+1), c_{i+1}) = \{ a_k|\ c_{i+1} < a_k,\ k=1,...,i+1\}$, and let $\Gamma(z(i+1),c_{i+1}) = \{ c_k|\ c_{i+1} < c_k,\ k=1,...,i+1\}$.  Since
 \begin{equation*}
 \{ a_1,....,a_{i+1}\} \setminus \{a_i, a_{i+1}\} = \{ c_1,....,c_{i+1}\} \setminus \{c_i, c_{i+1}\},
 \end{equation*}
if $c_{i+1} < a_{i+1}$, then $|\Gamma(x(i+1),c_{i+1})| = |\Gamma(z(i+1),c_{i+1})| +1$. Hence, if the position of $c_{i+1}$ in $\widetilde{z(i+1)}$ is $c_{\alpha_s}$, then $a_{\alpha_s} > c_{\alpha_s}$. This is a contradiction with $x(i+1) \leq_c z(i+1)$.

\textit{Case 1.2.} is similar; if $c_{i+1} > a_{i+1}=b_{i+1}$, then let $\Gamma(y(i+1),b_{i+1}) = \{ b_k|\ b_{i+1} < b_k,\ k=1,...,i+1\}$ and $\Gamma(z(i+1),b_{i+1}) = \{ c_k|\ b_{i+1} < c_k,\ k=1,...,i+1\}$. Since
 \begin{equation*}
 \{ b_1,....,b_{i+1}\} \setminus \{b_i, b_{i+1}\} = \{ c_1,....,c_{i+1}\} \setminus \{c_i, c_{i+1}\},
 \end{equation*}
$|\Gamma(z(i+1),b_{i+1})| = |\Gamma(y(i+1),b_{i+1})|+1$. Therefore, if the position of $b_{i+1}$ in $\widetilde{y(i+1)}$ is $b_{\alpha_{s'}}$, then $c_{\alpha_{s'}} > b_{\alpha_{s'}}$. This is a contradiction with $z(i+1) \leq_c y(i+1)$.

We proceed with \textit{Case 2.} that $b_{i+1}=a_{i+1} > b_i=a_j$. Once again, there are two subcases;

\textit{Case 2.1.} $c_{i+1} < a_{i+1}=b_{i+1}$, or

\textit{Case 2.2.}  $c_{i+1} > a_{i+1}=b_{i+1}$.

We continue with \textit{Case 2.1.}. Since,
 \begin{equation*}
 \{ a_1,....,a_{i+1}\} \setminus \{a_i, a_{i+1}\} = \{ c_1,....,c_{i+1}\} \setminus \{c_i, c_{i+1}\}.
 \end{equation*}
we have that $|\Gamma(x(i+1),a_{i+1})| \geq |\Gamma(z(i+1),a_{i+1})|+1 $. So, if the position of $a_{i+1}$ in $\widetilde{x(i+1)}$ is $a_{\alpha_s}$, then $a_{\alpha_s} > c_{\alpha_s}$. This is a contradiction with $x(i+1) \leq_c z(i+1)$.

Finally, we look at \textit{Case 2.2.} Since
\begin{equation*}
 \{ b_1,....,b_{i+1}\} \setminus \{b_i, b_{i+1}\} = \{ c_1,....,c_{i+1}\} \setminus \{c_i, c_{i+1}\},
 \end{equation*}
and since, $c_i \leq b_i < b_{i+1}$ we see that
$|\Gamma(z(i+1),b_{i+1})| = |\Gamma(y(i+1),b_{i+1})|+1$. Therefore, if the position of $b_{i+1}$ in $y(i+1)$ is $b_{\alpha_{s'}}$, then $c_{\alpha_{s'}} > b_{\alpha_{s'}}$. This is a contradiction with $z(i+1) \leq_c y(i+1)$.

We have dealt with all of the cases. We conclude that $c_{i+1}=a_{i+1}=b_{i+1}$. Notice that, as long as $a_k =b_k$ and $i < k < j$, the same arguments above work. Therefore, for any $k=i+1,...,j-1$ we have $c_k = a_k = b_k$.

Note also that $\widetilde{x(j)} = \widetilde{y(j)}$. By Remark \ref{R:deodhar0}, we know that $x(j) \leq_D z(j) \leq_D y(j)$. Hence, by
Lemma \ref{L:preparation4}, $\widetilde{x(j)} = \widetilde{y(j)}= \widetilde{z(j)}$.
Since $c_k = a_k = b_k$ for $k\notin \{i,j\}$, we either have that
$c_i = a_i,\ c_j=a_j$,\ or that $c_i = a_j,\ c_j = a_i$. Therefore, we either have that $z(j) = y(j)$, or that $z(j)=x(j)$.

Finally, for $k>j$, Lemma \ref{L:preparation5} shows that $c_k=a_k=b_k$. This shows that $z=y$ or $z=x$, hence $y$ covers $x$, and hence the proof is complete.
\end{proof}

\begin{rem}
Propositions \ref{P:Deodharcovering0} and  \ref{P:Deodharcovering1} show that a covering for the Pennell-Putcha-Renner ordering is a covering for the Deodhar ordering. Proposition \ref{P:Dcovers} below shows that the converse is also true.
\end{rem}

\begin{lem}\label{L:Dcovers1}
Let $x=(a_1,...,a_n),y=(b_1,...,b_n)\in R_n$. Suppose that there exists $i\in \{1,...,n-1\}$ such that
\begin{enumerate}
\item $a_k=b_k$ for $k=1,...,i-1$, and $b_i>a_i$,
\item $b_i = a_r$ for some $r>i$.
\end{enumerate}
Then, $y \rightarrow_D x$ implies that $y \rightarrow_{PPR} x$.
\end{lem}

\begin{proof}
Our strategy for proving that $y\rightarrow_D x$ implies $y\rightarrow_{PPR} x$ is as follows. We construct an element $z\in R_n$, such that $x\nleq_D z \leq_Dy $ and the pair $x,z\in R_n$ satisfy the hypothesis of the Proposition \ref{P:Deodharcovering1}. Thus, $z \rightarrow_D x$ implies that $\ell(z)= \ell(x)+1$, and this, by Lemma \ref{L:PPRcovering1} this implies that $z \rightarrow_{PPR} x$.
First, assume that $a_i=0$. Let $r'$ be the smallest index such that $i<r' \leq r$, and $a_{r'}$ is nonzero.
Define $z=(c_1,...,c_n)$ by setting $c_k=a_k$ if $k\notin \{i,r'\}$, and $c_i=a_{r'}$, $c_{r'}=a_i$. It is
easy to check that (see the proof of case $a_i>0$, below) $x \lneq_D z \leq_D y$, and that the pair $x,z$ satisfy the hypothesis of
Proposition \ref{P:Deodharcovering1}. Therefore, we are done in the case that $a_i=0$.
We proceed with the assumption that $a_i>0$.

Let $r'$ be the smallest integer such that
\begin{enumerate}
\item $i< r' \leq r$,
\item $a_i<a_{r'}$.
\end{enumerate}
Therefore,
\begin{equation}\label{E:observe1}
\text{if}\ i<s<r', \text{then}\ a_s < a_i.
\end{equation}

We define $z=(c_1,...,c_n)\in R_n$ as follows.  Let $k\in \{1,...,\widehat{i},...,\widehat{r'},....,n\}$. Set $c_k=a_k$. Also, set $c_i=a_{r'}$, and $c_{r'}=a_i$. It is easy to check that $x \lneq_D z$. We are going to show that $z \leq_D y$.  Note the following
\begin{enumerate}
\item $x(k)=y(k)=z(k)$ for $k=1,...,i-1$.
\item $\widetilde{x(i)} \leq_c \widetilde{z(i)} \leq_c \widetilde{y(i)}$.
\item $\widetilde{z(k)}=\widetilde{x(k)} \leq_c \widetilde{y(k)}$ for $k=r',...,n$.
\end{enumerate}

Therefore, it is enough to prove that $z(k) \leq_c y(k)$ for $k=i+1,...,r'-1$. To this end, $k\in \{i+1,...,r'-1\}$, and let $1\leq m \leq k$.  We are going to show that $|\Gamma(z(k),c_m)|\leq |\Gamma(y(k),c_m)|$.

There are two cases; $c_m < a_i$, or $c_m \geq a_i$. We start with the first one.

Since $c_m< a_i$, $m\notin \{i,r\}$, hence $a_m=c_m$.
The set of entries of $z(k)$ that are larger than $c_m=a_m$ is equal to the set of entries of $x(k)$ which are larger than $a_m$. Therefore,
\begin{equation}\label{E:case11}
|\Gamma(z(k),c_m)| = |\Gamma(x(k),c_m)|\leq |\Gamma(y(k),c_m)|,\ \text{if}\ c_m < a_i.
\end{equation}

The next case we check is that $c_m \geq a_i=c_{r'}$. By the observation (\ref{E:observe1}) above,
\begin{equation}
|\Gamma(z(k),c_m)|= |\Gamma(z(i),c_m)|.
\end{equation}

 On the other hand, since $z(i) \leq_c y(i)$,
 \begin{equation*}
 |\Gamma(z(i),c_m)| \leq |\Gamma(y(i),c_m)|,
 \end{equation*}
 and since $i<k$, we have
 \begin{equation*}
 |\Gamma(y(i),c_m)|\leq |\Gamma(y(k),c_m)|.
 \end{equation*}

Therefore,
\begin{equation}\label{E:case12}
|\Gamma(z(k),c_m)|\leq |\Gamma(y(k),c_m)|, \text{if}\ c_m \geq a_i.
\end{equation}

Hence, (\ref{E:case11}) and (\ref{E:case12}) shows that $z(k) \leq_c y(k)$ for
$k\leq r'-1$.
Having constructed $z\in R_n$, such that $x \lneq_D z \leq_D y$, since $y$ covers $x$ (in the Deodhar ordering), we have that $z=y$. Thus, we are exactly as in the hypotheses of the Proposition \ref{P:Deodharcovering1}. Therefore, we have that $\ell(y)= \ell(x)+1$, and that $y \rightarrow_{PPR} x$.

\end{proof}

\begin{lem}\label{L:Dcovers2}
Let $x=(a_1,...,a_n),y=(b_1,...,b_n)\in R_n$. Suppose that there exists $i\in \{1,...,n-1\}$ such that
\begin{enumerate}
\item $a_k=b_k$ for $k=1,...,i-1$, and $b_i>a_i$,
\item $b_i \notin \{a_1,...,a_n\}$.
\end{enumerate}
Then, $y \rightarrow_D x$ implies that $y \rightarrow_{PPR} x$.
\end{lem}

\begin{proof}
We make use of the following set
\begin{equation*}
\gamma(x,i)=\{ a_t:\ t>i\  a_i>a_t\}.
\end{equation*}

There are two cases; $\gamma(x,i)=\varnothing$, o r $\gamma(x,i)\neq \varnothing$. We start with the first case that  $\gamma(x,i)=\varnothing$.

Define $z=(c_1,...,c_n)$ as follows. Let $c_k=a_k$ for $k\neq i$, and let $c_i=b_i$.
Clearly $x \lneq_D z$. We are going to show that $z \leq_c y$.

It is enough to show that
\begin{eqnarray*}
|\Gamma(z(k),c_m)| \leq |\Gamma(y(k),c_m)|,
\end{eqnarray*}
for $k>i$, and $1\leq m \leq k$.

To this end, let $1\leq m \leq k$, and $i<k$.  If $c_m \geq a_i$, then

\begin{eqnarray*}
|\Gamma(z(k),c_m)| = |\Gamma(z(i),c_m)| = |\Gamma(y(i),c_m)| \leq |\Gamma(y(k),c_m)|.
\end{eqnarray*}

If $c_m < a_i$, then $c_m=a_m$, and
\begin{eqnarray*}
|\Gamma(z(k),c_m)|=|\Gamma(x(k),a_m)|\leq |\Gamma(y(k),a_m)|=|\Gamma(y(k),c_m)|.
\end{eqnarray*}
Therefore, if $\gamma(x,i)=\varnothing$, then $z\leq_D y$.

Having constructed $z\in R_n$, such that $x \lneq_D z \leq_D y$, since $y$ covers $x$ (in the Deodhar ordering), we have that $z=y$. Thus, we are exactly as in the hypotheses of the Proposition \ref{P:Deodharcovering1}. Therefore, we have that $\ell(y)= \ell(x)+1$, and that $y \rightarrow_{PPR} x$.

We continue with the case where $\gamma(x,i)\neq \varnothing$. Once again, there are two subcases;
either there exits $a_t\in \gamma(x,i)$ such that $b_i>a_t$, or for every $a_t \in \gamma(x,i)$, $a_t>b_i$.

We proceed with the first one.
Then, there exists $a_t\in\gamma(x,i)$ such that $b_i>a_t$. Let $t'$ be the smallest number such that
\begin{enumerate}
\item $i<t'$,
\item $a_i<a_{t'}<b_i$.
\end{enumerate}

Therefore, if $i<s<t'$, then
\begin{equation}\label{E:observe2}
a_i > a_s.
\end{equation}

Define $z=(c_1,...,c_n)$ as follows. If $k\notin \{i,t'\}$, then $c_k=a_k$, and $c_i=a_{t'}$, $c_{t'}=a_i$.
Clearly $x \lneq_D z$. We are going to show that $z \leq_c y$. It is enough to show that

\begin{enumerate}
\item $x(k)=y(k)=z(k)$ for $k=1,...,i-1$.
\item $\widetilde{x(i)} \leq_c \widetilde{z(i)} \leq_c \widetilde{y(i)}$.
\item $\widetilde{z(k)}=\widetilde{x(k)} \leq_c \widetilde{y(k)}$ for $k=t',...,n$.
\end{enumerate}

Therefore, it is enough to prove that $z(k) \leq_c y(k)$ for $k=i+1,...,t'-1$. To this end, $k\in \{i+1,...,t'-1\}$, and let $1\leq m \leq k$.  We are going to show that $|\Gamma(z(k),c_m)|\leq |\Gamma(y(k),c_m)|$.

There are two cases; $c_m < a_i$, or $c_m \geq a_i$. We start with the first one.

Since $c_m< a_i$, $m\notin \{i,t'\}$, hence $a_m=c_m$.
The set of entries of $z(k)$ that are larger than $c_m=a_m$ is equal to the set of entries of $x(k)$ which are larger than $a_m$. Therefore,
\begin{equation}\label{E:case21}
|\Gamma(z(k),c_m)| = |\Gamma(x(k),c_m)|\leq |\Gamma(y(k),c_m)|,\ \text{if}\ c_m < a_i.
\end{equation}

To deal with the other case we check that $c_m \geq a_i=c_{t'}$. By the observation (\ref{E:observe2}) above,
\begin{equation}
|\Gamma(z(k),c_m)|= |\Gamma(z(i),c_m)|.
\end{equation}

 On the other hand, since $z(i) \leq_c y(i)$,
 \begin{equation*}
 |\Gamma(z(i),c_m)| \leq |\Gamma(y(i),c_m)|,
 \end{equation*}
 and since $i<k$, we have
 \begin{equation*}
 |\Gamma(y(i),c_m)|\leq |\Gamma(y(k),c_m)|.
 \end{equation*}

Therefore,
\begin{equation}\label{E:case22}
|\Gamma(z(k),c_m)|\leq |\Gamma(y(k),c_m)|, \text{if}\ c_m \geq a_i.
\end{equation}

Hence, (\ref{E:case21}) and (\ref{E:case22}) show that $z(k) \leq_c y(k)$ for $k\leq t'-1$.

We proceed with the case that $\gamma(x,i)\neq \varnothing$, and $a_t>b_i$, for all $a_t\in \gamma(x,i)$.

Define $z=(c_1,...,c_n)$ as follows. If $k\neq i$, then $c_k=a_k$, and $c_i=b_i$.
Clearly $x \lneq_D z$. We are going to show that $z \leq_c y$.

It is enough to show that
\begin{eqnarray*}
|\Gamma(z(k),c_m)| \leq |\Gamma(y(k),c_m)|,
\end{eqnarray*}
for $k>i$, and $1\leq m \leq k$.

To this end, let $1\leq m \leq k$, and $i<k$.  If $c_m \geq b_i$, then

\begin{eqnarray*}
|\Gamma(z(k),c_m)| = |\Gamma(x(k),c_m)| \leq |\Gamma(y(i),c_m)|.
\end{eqnarray*}

If $c_m < b_i$, then $m<i$, and $c_m=a_m=b_m$. Note that the following. If $t>i$, then $b_t > b_i$. Assume otherwise. Let $i<t$ be the
smallest number such that $b_i>b_t$. Then,
\begin{eqnarray*}
|\Gamma(y(t),b_i)| < |\Gamma(x(k),b_i)|,
\end{eqnarray*}
which is a contradiction. Hence,
\begin{eqnarray*}
| \{ c_s:\ i<s\leq k,\ c_s>b_i\}| &=& | \{ b_s:\ i<s\leq k, b_s>b_i\}| = k-i+1
\end{eqnarray*}

Therefore,
\begin{eqnarray*}
|\Gamma(z(k),c_m)| &=& |\{c_s:\ i\geq s,\ c_s>c_m \}| + |\{ c_s:\ i<s\leq k,\ c_s>c_m\}|\\
&=& |\{b_s:\ i\geq s,\ b_s>c_m\}| + | \{ b_s:\ i<s\leq k,\  b_s>b_i\}| \\
&=& |\{b_s:\ i\geq s,\ b_s>c_m\}| + | \{ b_s:\ i<s\leq k,\  b_s>c_m\}| \\
&=& |\Gamma(y(k),c_m)|.
\end{eqnarray*}

Therefore, if $\gamma(x,i)\neq \varnothing$, then $z\leq_D y$. Having constructed $z\in R_n$, such that $x \lneq_D z \leq_D y$, since $y$ covers $x$ (in the Deodhar ordering), we have that $z=y$. Thus, we are exactly as in the hypotheses of the Proposition \ref{P:Deodharcovering1}. Therefore, we have that $\ell(y)= \ell(x)+1$, and that $y \rightarrow_{PPR} x$.

We have handled all the cases, and the proof is complete.

\end{proof}

\begin{prop}\label{P:Dcovers}
Let $x=(a_1,...,a_n)$ and $y=(b_1,...,b_n)$ be two elements from $R_n$. Suppose that $y \rightarrow_D x$. Then $y \rightarrow_{PPR} x$.
\end{prop}

\begin{proof}
Let $i\in \{1,....,n-1\}$ be the smallest index such that $k=1,...,i-1$, $a_k=b_k$ and $b_i > a_i$.

Then we have either

\textit{Case 1.} $b_i = a_r$ for some $r>i$, or

\textit{Case 2.} $b_i \notin \{a_1,...,a_n\}$.

Then, in the \textit{Case 1.}, the Lemma \ref{L:Dcovers1} shows that $y \rightarrow_{PPR} x$, and similarly, in the \textit{Case 2.}, the Lemma \ref{L:Dcovers2} shows that $y \rightarrow_{PPR} x$.
\end{proof}

\begin{thm}
The Deodhar ordering $\leq_D$ on $R_n$ is the same as Pennell-Putcha-Renner ordering
$\leq_{PPR}$ on $R_n$.
\end{thm}
\begin{proof}
By the Proposition \ref{P:Deodharcovering0}, and the Proposition
\ref{P:Deodharcovering1} we know that $y \rightarrow_{PPR} x$ implies  $y \rightarrow_D x$. Conversely, by the Proposition \ref{P:Dcovers}, if $y \rightarrow_D x$, then
$y \rightarrow_{PPR} x$.  Therefore, the two orderings have the same covering relations, hence they are the same order.
\end{proof}

\begin{cor}\label{C:deodhar} (Deodhar)
Let $x=(a_1,....,a_n)$ and $y= (b_1,...,b_n)$ be two permutations. Then, $x \leq y$ in the Bruhat ordering $\leq $ on $S_n$ if and only if  $x \leq_D y$ in the Deodhar ordering on $S_n$.
\end{cor}

\bibliographystyle{amsalpha}

\begin{thebibliography}{10}


\bibitem{MS05}
E. Miller, B. Sturmfels,
\emph{Combinatorial commutative algebra.} Graduate Texts in Mathematics, 227. Springer-Verlag, New York, 2005.

%\bibitem{BjornerBrenti}
%A . Bj\"{o}rner, F. Brenti,
%\emph{Combinatorics of Coxeter Groups.}
%Graduate Texts in Mathematic, 231. Springer, New York, (2005).

%\bibitem{Bourbaki}
%N. Bourbaki,
%\emph{Groupes et Algebr\`{e}s de Lie.}
%\'{E}l\'{e}ments de Math\'{e}matique, Chapitres 4,5 et 6. Hermann,
%Paris (1968).

%\bibitem{BRR89}
%D. Borwein, S. Rankin, L. Renner,
%\emph{Enumeration of injective partial transformations},
%Discrete Math.  73  (1989),  no. 3, 291--296.


%\bibitem{Carrell02}
%A. Bia{\l}ynicki-Birula, J.B. Carrell, W.M. McGovern
%\emph{Algebraic Quotients,Torus Actions and Cohomology,The Adjoint Representation and the Adjoint Action},
%Encyclopedia of Mathematical Sciences, vol.131,
%Subseries: Invariant Theory, vol.2, Springer-Verlag, 2002.


%\bibitem{Danilov78}
%V.I. Danilov,
%\emph{The geometry of toric varieties},
%Russian mathematical surveys 33, 97-154 (1978).



%\bibitem{Humphreys90}
%J. Humphreys,
%\emph{Reflection Groups and Coxeter Groups}
%Cambridge University Press (1990).

%\bibitem{GaRe86}
%A. Garsia, J. Remmel,
%\emph{$q$-rook configurations and a formula of Frobenius}
%J. of Combinatorial Theory (A) 41, 246-275 (1986).


%\bibitem{GorMac83}
%M. Goresky, R. MacPherson,
%\emph{Intersection Homology II}
%Invent. Math. 71, pp. 77-129 (1983)


%\bibitem{Kirwan88}
%F. Kirwan,
%\emph{Intersection Homology and Torus Actions}
%J.A.M.S., Vol. 1, No. 2. pp.385-400 (1988)


%\bibitem{KirWoolf}
%F. Kirwan, Woolf,
%\emph{An Introduction to Intersection Homology Theory}

%\bibitem{Li01}
%Z. Li,
%\emph{The Renner Monoids and Cell Decompositions of the Classical Algebraic Monoids}
%Ph.D. Thesis, University of Western Ontario, (2001).


%\bibitem{LLC06}
%Z. Li, Z. Li,  Y. Cao,
%\emph{Orders of the Renner Monoids},
%Journal of Algebra,. Vol. 301, No. 1, 344-359, 2006. 4.


%\bibitem{Moak81}
%D. Moak
%\emph{The $q$-Analogue of the Laguerre Polynomials}
%J. of Math. Analysis and Applications, Vol. 81, 20-47 (1981).

%\bibitem{Proctor82}
%R. A. Proctor,
%\emph{Classical Bruhat Orders and Lexicographic Shellabiliry.}
%Journal of Algebra \textbf{77} (1982), 104-126.

%\bibitem{Putcha88}
%M. Puthca,
% \emph{Linear Algebraic Monoids.}

\bibitem{PennellPutchaRenner}
E.A. Pennell M. Putcha, L. Renner,
\emph{Analogue of the Bruhat-Chevalley Order for Reductive Monoids.}
Journal of Algebra \textbf{196} (1997), 339-368.

\bibitem{Putcha01}
M. Putcha,
\emph{Shellability in Reductive Monoids.}
Tran. Amer. Math. Soc. 354 (2001), 413-426.

\bibitem{Renner86}
L.E. Renner,
\emph{Analogue of the Bruhat decomposition for algebraic monoids.}
Journal of Algebra \textbf{101} (1986), 303-338.


\bibitem{Renner95}
L.E. Renner,
\emph{Analogue of the Bruhat decomposition for algebraic monoids
II:
the length function and the trichotomy.}
Journal of Algebra \textbf{188} (1997), 272-291.



%\bibitem{Renner03}
%L.E. Renner,
%\emph{An explicit cell decomposition of the canonical compactification of an algebraic group},
%Can. Math. Bull. 46. 140-148 (2003).

%\bibitem{Renner04}
%L. Renner,
%\emph{Linear Algebraic Monoids.}
%Encyclopedia of Mathematical Sciences, vol.134,
%Subseries: Invariant Theory, vol.5, Springer-Verlag, 2005.

%\bibitem{Renner07}
%L. Renner
%\emph{The $H$-polynomial of a Semisimple Monoid}
%Preprint (2007).

%\bibitem{Sol95}
%L. Solomon,
%\emph{An introduction to reductive monoids.}

%\bibitem{Stan87}
%R. Stanley,
%\emph{Generalized $h$-Vectors, Intersection Cohomology of Toric Varieties, and Related Results}
%Advanced Studies in Pure Mathematics 11, 1987
%Commutative Algebra and Combinatorics 187-213.

%\bibitem{StanVol1}
%R. Stanley,
%\emph{Enumerative combinatorics, volume 1.}


\end{thebibliography}

\end{document}